\documentclass[11pt,leqno]{article}
\usepackage{amsmath,amssymb,amsthm,latexsym}
\usepackage{indentfirst}
\usepackage{amsmath}

\newtheorem{theorem}{\bf \large theorem}[section]
\newtheorem{proposition}{\bf \large Proposition}[section]
\newtheorem{corollary}{\bf \large Corollary}[section]

\newtheorem{lemma}{\bf\large Lemma}[section]
\newtheorem{Conjecture}{\bf\large Conjecture}[section]

\title{\textbf{The second gap for  self-shrinkers with constant norm of the second fundamental
form}}
\author {\normalsize { Pengpeng Cheng,~~Tongzhu Li}\\
\small{Department of Mathematics, Beijing Institute of
Technology,}\\
\small{Beijing, 100081, China.}\\
\small{E-mail:~3120235951@bit.edu.cn,~~litz@bit.edu.cn.}}

\date{}

\begin{document}
\maketitle
\begin{abstract}
Let $X: M^{n}\to \mathbb{R}^{n+1}$ be a complete self-shrinker with constant squared norm of the second fundamental
form $S$. In this paper, we prove that if $S\leq \frac{10}{7}$, then $S=1$ or $S=0$, and the self-shrinker is isometric to either a round sphere $\mathbb{S}^n(\sqrt{n})$ with the center at the origin, or a cylinder $\mathbb{S}^k(\sqrt{k})\times \mathbb{R}^{n-k},~~1\leq k\leq n-1$, or a plane through the origin.
\end{abstract}

\vskip 1 cm
\section{Introduction}
Let $X: M^{n}\to \mathbb{R}^{n+1}$ be a smooth immersed  hypersurface in the $(n+1)$-dimensional Euclidean space $\mathbb{R}^{n+1}$. We call $X$ a self-shrinker if
\begin{equation}\label{self}
H+\langle X,\xi\rangle=0,
\end{equation}
where $H$ and $\xi$
denote the mean curvature and the unit normal vector of $X$, respectively.

An important problem in the study of the mean curvature flow is to understand the geometry of the possible singularities of the flow.
Self-shrinkers  describe all possible blow ups at a given type I singularity of the mean curvature flow. On the other hand, self-shrinkers correspond to self-shrinking solutions to the mean curvature flow. Therefore it is  important in the study of the mean curvature flow to classify all self-shrinkers.

There are many important results about the classification of self-shrinkers. Abresch and Langer in \cite{ab} classified all closed self-shrinker curves in $\mathbb{R}^2$, which are called Abresch-Langer curves, and they showed that the only simple closed self-shrinker is the circle. In \cite{e-h}, Ecker and Huisken proved that if an entire graph with polynomial volume growth is a self-shrinker, then it is necessarily a hyperplane, which is proved by  Wang without the assumption of polynomial volume growth in \cite{w}. In \cite{hui-1}, Huisken proved that sphere $\mathbb{S}^n(\sqrt{n})$ is the only $n$-dimensional compact self-shrinker in $\mathbb{R}^{n+1}$ with non-negative mean curvature. When the hypersurface is complete, in \cite{hui-2}, Huisken proved that the complete self-shrinkers with non-negative mean curvature and polynomial volume growth are $\Gamma\times \mathbb{R}^{n-1}$, or $\mathbb{S}^m(\sqrt{m})\times \mathbb{R}^{n-m},~(0\leq m\leq n)$ if the norm of the second fundamental form is bounded, where $\Gamma$ is an Abresch-Langer curve. In \cite{cold-2}, Colding and Minicozzi proved that Huisken's classification theorem still holds without the assumption that $S$ is bounded. Furthermore Colding and Minicozzi gave a classification of $n$-dimensional $\mathfrak{F}$-stable complete self-shrinkers with polynomial volume growth and parallel principal normal, which is  extended to higher codimensional submanifolds by Andrews,  Li and  Wei \cite{and}. In \cite{caoli}, Cao and Li classified the complete self-shrinkers with polynomial volume growth and $S\leq 1$. In the works of Ecker-Huisken \cite{e-h} and Colding-Minicozzi \cite{cold-2} the polynomial volume growth plays an important role
in studying self-shrinkers. We say the hypersurface $X(M^n)$ in $\mathbb{R}^{n+1}$ has polynomial volume growth if there exist constants $C$ and $d$ such that for all $r\geq 1$, there holds
$$Vol\Big(B(r)\cap X(M^n)\Big)\leq Cr^d,$$
where $B(r)$ denotes a Euclidean ball with radius $r$.
According to the results of Cheng and Zhou \cite{chzh}, Ding and Xin \cite{d-2}, the self-shrinker has polynomial volume growth if and only if the self-shrinker is proper.

According to the results of  \cite{hall}, Halldorsson constructed a complete self-shrinker $\gamma\times \mathbb{R}^{n-1}$ without polynomial volume growth, where $\gamma$ is a complete self-shrinker curve of Halldorsson.
By the generalized maximum principle for the $L$-operator proved by Cheng and Peng in \cite{c-2}, Cheng and Peng studied complete self-shrinkers without the assumption of polynomial volume growth. In \cite{d-3}, Ding and Xin proved that a $2$-dimensional complete self-shrinker with polynomial volume growth and constant squared norm $S$ of the second fundamental form is one of $\mathbb{R}^2$, $\mathbb{S}^1(1)\times \mathbb{R}$ and $S^2(\sqrt{2})$. Cheng and Ogata \cite{c-1} showed that Ding and Xin's result holds without the assumption of polynomial volume growth:

\begin{theorem}\cite{c-1}
A $2$-dimensional complete self-shrinker $X: M^2\to \mathbb{R}^3$ with constant squared norm of the second fundamental form is isometric to one of $\mathbb{R}^2$, $\mathbb{S}^1(1)\times \mathbb{R}$ and $\mathbb{S}^2(\sqrt{2})$.
\end{theorem}

For $3$-dimensional hypersurfaces,  Cheng, Li and Wei \cite{c-3} gave the same results under the assumption that $f_4$ is constant.
Based on the above results, Cheng and Wei \cite{cw-survey}  proposed the following conjecture:
\begin{Conjecture}\label{con}
An $n$-dimensional complete self-shrinker $X: M^n\to \mathbb{R}^{n+1}$ with constant squared norm of the second fundamental form is isometric to one of a hyperplane $\mathbb{R}^n$ passing through the origin,  a round sphere $\mathbb{S}^n(\sqrt{n})$ with the center at the origin, or a cylinder $\mathbb{S}^k(\sqrt{k})\times \mathbb{R}^{n-k},~1\leq k\leq n$. In particular, $S$ must be $0$ or $1$.
\end{Conjecture}
In \cite{c-6}, Cheng-Li-Wei proved the Conjecture for dimension $3$ under $f_3=$constant:
\begin{theorem}(\cite{c-6})
Let $X: M^3\to \mathbb{R}^4$
be a $3$-dimensional complete self-shrinker in $R^4$. If the squared norm
$S$ of the second fundamental form and $f_3$ are constant, then $X: M^3\to \mathbb{R}^4$
is isometric to one of
(1) $\mathbb{S}^3(\sqrt{3})$,
(2) $\mathbb{R}^3$,
(3) $\mathbb{S}^1(1)\times \mathbb{R}^2$, and
(4) $\mathbb{S}^2(\sqrt{2})\times \mathbb{R}^1$.\\
In particular, $S$ must be $0$ or $1$; $f_3$ must be $0, 1,\frac{\sqrt{2}}{2}$ or $\frac{\sqrt{3}}{3}$,
 where $h_{ij}$'s are the components of the second
fundamental form, $S=\sum_{i,j}h_{ij}^2$ and $f_3=\sum_{i,j,k}h_{ij}h_{jk}h_{ki}$.
\end{theorem}

Under a polynomial volume growth assumption \cite{caoli}, in \cite{c-0} Cheng and Wei proved the Conjecture when the constant satisfies  $S\le10/7$. Recently, Cheng, Li, and Wei proved the Conjecture when  $S\le7/5$ in every dimension without a polynomial volume growth assumption \cite{cw-1}.
In this paper, we have improved the upper bound of the constant in \cite{cw-1}, but it is not optimal. In detail, we have obtained the following results.

\begin{theorem}\label{thm1}
Let $X:M^n\to\mathbb R^{n+1}$ be a  complete self-shrinker. If the squared norm of the second fundamental form $S$ is constant  and  satisfies $S\le10/7$, then $M^n$ is one of the following:\\
(1) a hyperplane $\mathbb{R}^n$ passing through the origin,\\
(2) a round sphere $\mathbb{S}^n(\sqrt{n})$ with the center at the origin,\\
(3) a cylinder $\mathbb{S}^k(\sqrt{k})\times \mathbb{R}^{n-k},~~1\leq k\leq n-1.$
\end{theorem}

By combining Theorem \ref{thm1} with the Gauss equation and the generalized maximum principle, we obtain the same classification under a lower bound for the scalar curvature.

\begin{corollary}\label{cor1}
Let $X:M^n\to\mathbb R^{n+1}$ be a  complete self-shrinker. If the squared norm  of the second fundamental form  $S$ is constant
and the scalar curvature  satisfies $R\ge-10/7$,  then $M^n$ is one of the following:\\
(1) a hyperplane $\mathbb{R}^n$ passing through the origin,\\
(2) a round sphere $\mathbb{S}^n(\sqrt{n})$ with the center at the origin,\\
(3) a cylinder $\mathbb{S}^k(\sqrt{k})\times \mathbb{R}^{n-k},~~1\leq k\leq n-1.$
\end{corollary}

The proof follows the maximum-principle approach of Cheng, Li, and Wei. We apply the generalized maximum principle for the $ L$-operator to a function of $f_3=\sum_i\lambda_i^3$ and $f_4=\sum_i\lambda_i^4$, where $\lambda_i$ are the principal curvatures. At the resulting limiting configuration, we combine the estimate for the second covariant derivatives of the second fundamental form in \cite{cw-1} with an algebraic inequality for the normalized principal curvatures. These estimates give a contradiction whenever $7/5<S\le10/7$.

This paper is organized as follows. In section 2, we give the necessary facts about  the self-shrinkers and the generalized maximum principle.
In section 3, we present some  estimates and the proofs of the main results.

\vskip 1 cm
\section{Fundamental equations of self-shrinker}
Let $X: M^n\to \mathbb{R}^{n+1}$ be an $n$-dimensional immersed  hypersurface in the  Euclidean space $\mathbb{R}^{n+1}$. Let $\{e_1,\cdots,e_n\}$ be a local orthonormal frame field of $X(M^n)$ with respect to the first fundamental form $I=\langle dX,dX\rangle$ with dual coframe field $\{\omega_1,\cdots,\omega_n\}$. Let $\xi$ be the unit normal vector field of $X$.
In this paper we use the following conventions on the ranges of indices:
$ 1\leq i,j,k,l\leq n.$
Thus we have the structure equation of $X$,
\begin{equation}\label{eq-str}
\begin{split}
&dX=\sum_{i} \omega_{i}e_{i},\\
&de_{i}=\sum_{j} \omega_{ij}e_{j}+\sum_jh_{ij}\omega_j\xi,\\
&d\xi=-\sum_{i,j}h_{ij}\omega_je_{i},
\end{split}
\end{equation}
where $h_{ij}$ denote the components of the second fundamental form of $X$
 and $\omega_{ij}$ are the Levi-Civita connection forms of $I$ with respect to $\{\omega_1,\cdots,\omega_n\}$.

The Levi-Civita connection $\nabla$ is defined by $\nabla e_i=\sum_j\omega_{ij}\otimes e_j$ and $\omega_{ij}=-\omega_{ji}$.
The symmetric $2$-form
$$ h=\sum_{i,j} h_{ij}\omega_{i}\otimes \omega_{j} $$
is called the second fundamental form of the hypersurface $X$. Let $S=\sum_{i,j}(h_{ij})^{2}$ be the squared norm of the second fundamental form of $X$. $H=\sum_{i} h_{ii}$ is called the mean curvature of $X$, which is different from the definition in differential geometry books. The Gauss equations and the Codazzi equations of the hypersurface $X$ are given by
\begin{equation}\label{eq-1}
\begin{split}
&R_{ijkl}=h_{ik}h_{jl}-h_{il}h_{jk},~~1\leq i,j,k,l\leq n,\\
&h_{ij,k}=h_{ik,j}, ~~1\leq i,j,k\leq n,
\end{split}
\end{equation}
where the covariant derivative of $h_{ij}$ is defined by
$$\sum_kh_{ij,k}\omega_k=dh_{ij}+\sum_kh_{ik}\omega_{kj}+\sum_kh_{kj}\omega_{ki}.$$
By the Gauss equation, the scalar curvature $R$ satisfies
\begin{equation}\label{eq-2}
R=H^2-S.
\end{equation}

We define the invariants  $f_3$  and $f_4$ as follows:
$$f_3=\sum_{i,j,k}h_{ij}h_{jk}h_{ki},~~f_4=\sum_{i,j,k,l}h_{ij}h_{jk}h_{kl}h_{li}.$$
We call $f_3$ the $3$-mean curvature of the hypersurface $X$.

Let $\lambda_1,\cdots,\lambda_n$ be the principal curvatures of $X$, then
$$H=\sum_i\lambda_i,~~S=\sum_i\lambda^2_i,~~f_3=\sum_i\lambda^3_i, ~~f_4=\sum_i\lambda^4_i.$$

Let $f$ be  a smooth function on $M^n$, we define the covariant derivatives $f_i$ and $f_{ij}$  of $f$ as follows,
\begin{equation}\label{f-f}
df=\sum_{i}f_{i}\omega_{i},~~~\sum_{j}f_{ij}\omega_j=df_{i}+\sum_{j} f_{j}\omega_{ji}.
\end{equation}
The gradient $\nabla f$ of $f$ is the vector field dual to the $1$-form $df$, which is defined by
$\nabla f=\sum_if_ie_i.$
$$\Delta f=\sum_{i} f_{ii}.$$
The following  important operator $L$  was introduced by Colding and Minicozzi (see \cite{cold-2}),
$$Lf=\Delta f-\langle X,\nabla f\rangle.$$

Now let $X:M^n\to \mathbb{R}^{n+1}$ be a
self-shrinker, then
\begin{equation}\label{lself}
H+\langle X,\xi\rangle=0.
\end{equation}
 Using equation (\ref{lself}), we have the following equations,
\begin{equation}\label{lself-1}
\begin{split}
&H_i=\sum_kh_{ik}\langle X,e_k\rangle,~~1\leq i\leq n,\\
&H_{ij}=\sum_kh_{ik,j}\langle X,e_k\rangle+h_{ij}-H\sum_kh_{ik}h_{kj}.
\end{split}
\end{equation}
By a direct calculation, or see \cite{c-5}, we have the following equations:
\begin{lemma}\label{le-4-1}
Let $X: M^n\to \mathbb{R}^{n+1}$ be a self-shrinker in $\mathbb{R}^{n+1}$, then we have
\begin{equation}\label{eq-3}
\begin{split}
&Lh_{ij}=(1-S)h_{ij},\\
&LH=H(1-S),\\
&\frac{1}{2}LH^2=|\nabla H|^2+H^2(1-S),\\
&\frac{1}{2}LS=\sum_{i,j,k}h^2_{ij,k}+S(1-S).
\end{split}
\end{equation}
\end{lemma}
By applying the product rule
to the traces and then using \eqref{eq-3}, we obtain the following equations:
\begin{equation}\label{eq-4}
 \begin{aligned}
 \frac13 L f_3&=(1-S)f_3+2\sum_{i,j,k=1}^n\lambda_i h_{ijk}^2,\\
 \frac14L f_4&=(1-S)f_4+2A+B,
 \end{aligned}
\end{equation}
where \[
 A=\sum_{i,j,k=1}^n\lambda_i^2h_{ijk}^2,\qquad
  B=\sum_{i,j,k=1}^n\lambda_i\lambda_jh_{ijk}^2.
\]
We shall use the following generalized maximum principle for the
$L$-operator, due to Cheng and Peng \cite{c-2}.

\begin{lemma}[Generalized maximum principle]\label{lem:maximum}
Let $X:M^n\to\mathbb R^{n+1}$ be a complete self-shrinker whose Ricci curvature is bounded from below. If $u\in C^2(M)$ is bounded above, then there exists a sequence of points $p_j\in M$ such that
\[
 u(p_j)\longrightarrow\sup_{M^n} u,\qquad
  |\nabla u|(p_j)\longrightarrow0,\qquad
  \limsup_{j\to\infty}L u(p_j)\le0.
\]
\end{lemma}

\section{Proof of Main theorem}
Let $X:M^n\to\mathbb R^{n+1}$ be a complete self-shrinker whose $S$ is constant.
At any point $p\in M^n$, we can choose an orthonormal basis $\{e_1,e_2,\cdots,e_n\}$  such that
	$$(h_{ij})=\operatorname{diag}(\lambda_1,\lambda_2,\cdots,\lambda_n).$$

The components $h_{ijk\ell}$ of its second covariant derivative are defined by
\begin{equation}\label{eq:hijkl-definition}
 \sum_{\ell=1}^n h_{ijk\ell}\,\omega_\ell
 =dh_{ijk}+\sum_{\ell=1}^n h_{\ell jk}\,\omega_{\ell i}
 +\sum_{\ell=1}^n h_{i\ell k}\,\omega_{\ell j}
 +\sum_{\ell=1}^n h_{ij\ell}\,\omega_{\ell k}.
\end{equation}
With this convention, the Ricci  identity is
\begin{equation}\label{eq:ricci-commutation}
 h_{ijk\ell}-h_{ij\ell k}
 =\sum_{m=1}^n h_{mj}R_{m i k\ell}
 +\sum_{m=1}^n h_{im}R_{m j k\ell}.
\end{equation}
Since $S$ is constant, we have
\begin{equation}\label{eq:3.5}
	\begin{aligned}
		&\sum_{i,j,k}h_{ijk}^2= S(S-1),\ \ \mathcal L h_{ij}=(1-S)h_{ij},\\
		&\sum_{i,j,k,l}h_{ijkl}^2=S(S-1)(S-2) +3(A-2B).
	\end{aligned}
\end{equation}
Defining $u_{ijkl}$ by
$$
u_{ijkl}:=\frac 14(h_{ijkl}+h_{jkli}+h_{klij}+h_{lijk} ),
$$
in view of the Ricci identity (\ref{eq:ricci-commutation}),  a direct computation yields
\begin{equation}\label{eq:3.6}
	\begin{aligned}
		\sum_{i,j,k,l}h_{ijkl}^2 = \sum_{i,j,k,l}u_{ijkl}^2 +\frac32\bigl( Sf_4-f_3^2\bigl).
	\end{aligned}
\end{equation}
Combining \(\eqref{eq:3.5}\) and \(\eqref{eq:3.6}\), we obtain
\begin{equation}\label{eq:12}
	\frac32Sf\le S(S-1)(S-2)+3(A-2B).
\end{equation}

In order to prove Theorem~\ref{thm1}, we first give several estimates
that will be used throughout the proof. The following pointwise estimates can be found in Cheng and Wei \cite{c-0}.
\begin{lemma}[Pointwise estimates]\label{lem:pointwise}
Let $X:M^n\to\mathbb R^{n+1}$ be a self-shrinker, whose $S$ is constant
and satisfies $S>1$. At any point of $M^n$, choose an
orthonormal frame in which the second fundamental form is diagonal, and
write $$\lambda_1=\max_i\lambda_i,~~\lambda_2 =\min_i\lambda_i.$$
Let $$f=f_4-\frac{f_3^2}{S}.$$Then the following inequalities hold:
\begin{align}
 &f\ge\frac{(\lambda_1-\lambda_2)^2\lambda_1^2\lambda_2^2}
                 {\lambda_1^2+\lambda_2^2}\ge 0,\label{eq:extreme-f}\\
 &A-B\le\frac13(\lambda_1-\lambda_2)^2S(S-1)(1-\alpha),\label{eq:ABbound}\\
 &\sum_k\left(\sum_i\lambda_i^2h_{iik}\right)^2
 \le\frac{1+2\alpha}{3}S(S-1)f,\label{eq:derivative-bound}\\
 &\left(\sum_{i,j,k}\lambda_i h_{ijk}^2\right)^2
 \le S(S-1)\left\{\frac{A+2B}{3}
       -\frac43\sum_k\frac{\left(\sum_i\lambda_i^2h_{iik}\right)^2}{S+2\lambda_k^2}\right\},\label{eq:tensor-bound}
\end{align}
where  $\alpha=\dfrac{\sum_ih_{iii}^2 }{S(S-1)}$.
\end{lemma}

We also use the following estimate of Cheng--Li--Wei
\cite[Proposition 3.1]{cw-1} for the second covariant derivatives
of the second fundamental form.
\begin{proposition}\label{pro3.1} For an $n$-dimensional  self-shrinker  $X: M\rightarrow \mathbb{R}^{n+1}$
	with constant  squared norm of the second fundamental form and $S>1$, we have
	\begin{equation}\label{eq:fourth}
		\begin{aligned}
			&S(S-1)(S-2) +3(A-2B) \\
			&\ge 2 Sf-2A+\frac 4{S(S-1)}\sum_i\lambda_i^2\big (\sum_j\lambda_j^2h_{jji}\big)^2\\
			&+\frac {2f_3}S \sum_{i,j,k}\lambda_ih_{ijk}^2  +\frac 1{S^2}\big(\sum_{i,j,k}\lambda_ih_{ijk}^2 \big)^2.
		\end{aligned}
	\end{equation}
\end{proposition}

The preceding estimates control the contractions involving the covariant
derivatives of the second fundamental form. To complete the gap argument, we
also need an elementary inequality relating the spread of the normalized
principal curvatures to the quantity $f$. We record this inequality in the
following lemma.
\begin{lemma}\label{lem-spectral}
Let $\mu_1,\cdots,\mu_n\in \mathbb{R}$ satisfy $\sum_i\mu_i^2=1$, and put
\[
 G=\sum_i\mu_i^4-\left(\sum_i\mu_i^3\right)^2,
 \qquad \delta=(\max_i\mu_i-\min_i\mu_i)^2.
\]
Then $G\ge0$, $0\le\delta\le2$, and
\begin{equation}\label{eq:spectral}
 \delta<\frac65+\frac{15}{4}G.
\end{equation}
\end{lemma}

\begin{proof}
Since $\sum_i\mu_i^2=1$, we obtain
\begin{equation}\label{eq:variance}
 G=\sum_i\mu_i^2\left(\mu_i-\sum_j\mu_j^3\right)^2\ge0.
\end{equation}
 For any two distinct indices $i$ and $j$, we have
\[
 (\mu_i-\mu_j)^2\le2(\mu_i^2+\mu_j^2)\le2,
\]
which implies that $0\le\delta\le2$. Moreover, $\delta=0$ occurs precisely when all the
$\mu_i$ are equal.

If $\delta\le1$, then $G\ge0$ implies
\[	\delta<\frac65+\frac{15}{4}G,\]
which proves \eqref{eq:spectral} in this case.

Suppose now that $\delta>1$. Set $a=\max_i\mu_i$, $b=\min_i\mu_i$, and
$x=a^2+b^2$. Since \(a\) and \(b\) are attained at distinct indices, we have $0<x\le1$. From \eqref{eq:variance}, we obtain
\[
 G\ge a^2\left(a-\sum_i\mu_i^3\right)^2
       +b^2\left(b-\sum_i\mu_i^3\right)^2.
\]
For an arbitrary $z\in \mathbb{R}$, we have
\begin{equation}\label{eq:0}
 a^2(a-z)^2+b^2(b-z)^2
=x\left(z-\frac{a^3+b^3}{x}\right)^2
+\frac{a^2b^2(a-b)^2}{x}
\ge\frac{a^2b^2(a-b)^2}{x}.
\end{equation}
Then taking $z=\sum_i\mu_i^3$ in \eqref{eq:0} and using $(a-b)^2=\delta$ together with
$2ab=x-\delta$, we obtain
\begin{equation}\label{eq:two-root}
 G\ge\frac{a^2b^2(a-b)^2}{x}
     =\frac{\delta(\delta-x)^2}{4x}.
\end{equation}
For fixed $\delta>1$, the derivative of $(\delta-x)^2/x$ with respect to $x$ is
$1-\delta^2/x^2<0$ on $(0,1]$. Hence the right-hand side of
\eqref{eq:two-root} is minimized at $x=1$, and therefore
\begin{equation}\label{eq:variance-diameter}
 G\ge\frac{\delta(\delta-1)^2}{4}.
\end{equation}
Using \eqref{eq:variance-diameter} yields
\[
\begin{aligned}
 \frac65+\frac{15}{4}G-\delta
 &\ge \frac{75\delta^3-150\delta^2-5\delta+96}{80}\\
 &=\frac{15}{16}\left(\delta-\frac{27}{20}\right)^2
      \left(\delta+\frac7{10}\right)
   +\frac{\delta}{1280}+\frac{51}{12800}>0.
\end{aligned}
\]
This proves \eqref{eq:spectral} for $\delta>1$ and completes the proof.
\end{proof}
\begin{proof}[Proof of Theorem~\ref{thm1}]
Since $S$ is constant, the Gauss equation gives a lower bound for the Ricci
curvature. Moreover, the quantities $\lambda_i$, $h_{ij}$, $h_{ijk}$, and $h_{ijk\ell}$
are bounded according to \eqref{eq:3.5}. Thus, for a fixed constant $c>0$, the function
\[
F=\frac14Sf_4-\frac16cf_3^2
\]
is bounded. In view of \eqref{eq-4}, we obtain
\begin{equation}\label{eq:Fc-drift}
\begin{aligned}
 L F
={}&S(1-S)f_4+S(2A+B)\\
&-c\left\{(1-S)f_3^2
+2f_3\sum_{i,j,k}\lambda_i h_{ijk}^2
+3\sum_k\left(\sum_i\lambda_i^2h_{iik}\right)^2
\right\}.
\end{aligned}
\end{equation}
By applying the generalized maximum principle for the $\mathcal L$-operator
to the function $F$, there exists a sequence $\{p_m\}\subset M$ such that
 $$
\lim_{m\to\infty}F(p_m)=\sup F,
\ \ \lim_{m\to\infty}|\nabla F|(p_m)=0,
$$
$$
\limsup_{m\to\infty} \mathcal LF(p_m)\leq 0.
$$
Since $\lambda_i$,  $h_{ijk}$ and  $h_{ijkl}$  are bounded, we can assume that
$\{\lambda_i(p_m)\}$, $\{h_{ijk}(p_m)\}$ and  $\{h_{ijkl}(p_m)\}$  converge
if necessary  by taking a subsequence of $\{p_m\}$. By  abusing  notation, we still use
$\lambda_i$,  $h_{ijk}$ and  $h_{ijkl}$ to denote
limits of $\{\lambda_i(p_m)\}$, $\{h_{ijk}(p_m)\}$ and  $\{h_{ijkl}(p_m)\}$, respectively and all computations are processed under limits.
Passing to the limit in \eqref{eq:Fc-drift} along this subsequence
and using $\limsup_{m\to\infty}\mathcal LF(p_m)\le0$, we obtain
\begin{equation}\label{eq:maxineq}
 -(S-1)(Sf_4-cf_3^2)\le2cf_3\sum_{i,j,k}\lambda_i h_{ijk}^2-S(2A+B)
 +3c\sum_k\left(\sum_i\lambda_i^2h_{iik}\right)^2.
\end{equation}

If $S\le1$, then constancy of $S$ and \eqref{eq-3} give
$0=|\nabla h|^2+(1-S)S$. Hence $S=0$ or $S=1$ and $\nabla h=0$.
The classification of complete self-shrinkers with parallel second
fundamental form gives the hyperplane through the origin, the sphere
$S^n(\sqrt n)$, or a generalized cylinder
$S^k(\sqrt k)\times\mathbb R^{n-k}$, $1\le k\le n-1$, or see
\cite{chengli}. Thus the asserted classification holds in this case. We
henceforth assume $S>1$.

The range $1<S\le7/5$ has already been excluded by Cheng--Li--Wei \cite[Theorem 3.1]{cw-1}.
Next, we will prove that $7/5<S\le10/7$ does not occur. Namely, we will prove $S>10/7$ if $S>7/5$. This is a contradiction since $S$ is constant.

It is convenient to express the range of $S$ in terms of
$t=(S-1)/S$. We work on the closed range
\begin{equation}\label{eq:interval}
 \frac75\le S\le\frac{10}{7},\qquad \frac27\le t\le\frac3{10},
\end{equation}
because all coefficients used below extend continuously to the endpoint
$S=7/5$.

We now estimate the three contraction terms that occur in
\eqref{eq:fourth}. Define
\begin{equation}\label{eq:Phi-expanded}
\begin{aligned}
 \Phi={}&-\frac{13}{3}tSf
 -\frac{2f_3}{S}\sum_{i,j,k}\lambda_i h_{ijk}^2
 -\frac1{S^2}\left(\sum_{i,j,k}\lambda_i h_{ijk}^2\right)^2\\
 &-\frac4{tS^2}\sum_k\lambda_k^2
       \left(\sum_i\lambda_i^2h_{iik}\right)^2 .
\end{aligned}
\end{equation}
In \eqref{eq:maxineq}, we choose $c=1+c_1(t)$, where
\begin{equation}\label{eq:c}
 c_1(t)=\frac{22+3t-\sqrt{(22+3t)^2-400}}{26}.
\end{equation}
The value of $c_1(t)$ in \eqref{eq:c} satisfies
\begin{equation}\label{eq:c1}
 \left(\frac{10}{13}+c_1\right)^2=\frac{3(14+t)}{13}c_1.
\end{equation}
Substituting $c=1+c_1(t)$ into \eqref{eq:maxineq} and using
the definition \eqref{eq:Phi-expanded} of $\Phi$, we obtain
\[
\begin{aligned}
 \Phi\le{}&-\frac{13}{3}(2A+B)-\frac{13c_1t}{3}f_3^2
 +\left(\frac{10}{3}+\frac{13c_1}{3}\right)
       \frac{2f_3}{S}\sum_{i,j,k}\lambda_i h_{ijk}^2\\
 &+\frac{13(1+c_1)}{S}\sum_k\left(\sum_i\lambda_i^2h_{iik}\right)^2
 -\frac1{S^2}\left(\sum_{i,j,k}\lambda_i h_{ijk}^2\right)^2
 -\frac4{tS^2}\sum_k\lambda_k^2
       \left(\sum_i\lambda_i^2h_{iik}\right)^2\\
 \le{}&-\frac{13}{3}(2A+B)
 +\left\{\frac{13}{3tc_1}\left(\frac{10}{13}+c_1\right)^2-1\right\}
       \frac1{S^2}\left(\sum_{i,j,k}\lambda_i h_{ijk}^2\right)^2\\
 &+\frac{13(1+c_1)}{S}\sum_k\left(\sum_i\lambda_i^2h_{iik}\right)^2
 -\frac4{tS^2}\sum_k\lambda_k^2
       \left(\sum_i\lambda_i^2h_{iik}\right)^2,
\end{aligned}
\]
where we used $2xy\leq x^2+y^2$ in the last inequality with
\[
\begin{aligned}
 x=\sqrt{\frac{3}{13tc_1}}\left(\frac{10}{3}+\frac{13c_1}{3}\right)
       \frac1S\sum_{i,j,k}\lambda_i h_{ijk}^2,\quad
 y=\sqrt{\frac{13tc_1}{3}}f_3.
\end{aligned}
\]
Applying \eqref{eq:tensor-bound} to the term containing
$\left(\sum_{i,j,k}\lambda_i h_{ijk}^2\right)^2$ in this upper bound
for $\Phi$, we obtain
\[
\begin{aligned}
 \Phi\le{}&-\frac{13}{3}(2A+B)
 +\frac{13(1+c_1)}{S}\sum_k\left(\sum_i\lambda_i^2h_{iik}\right)^2
 -\frac4{tS^2}\sum_k\lambda_k^2
       \left(\sum_i\lambda_i^2h_{iik}\right)^2\\
 &+\left\{\frac{13}{3c_1}\left(\frac{10}{13}+c_1\right)^2-t\right\}
 \left\{\frac13(A+2B)
 -\frac43\sum_k\frac{\left(\sum_i\lambda_i^2h_{iik}\right)^2}
                         {S+2\lambda_k^2}\right\}\\
 ={}&-4A+5B+
 \sum_k\left(\frac{13(1+c_1)}{S}-\frac{4\lambda_k^2}{tS^2}
 -\frac{56}{3(S+2\lambda_k^2)}\right)
 \left(\sum_i\lambda_i^2h_{iik}\right)^2,
\end{aligned}
\]
where the last equality follows from \eqref{eq:c1}.
Because of
\[
 -\frac{4\lambda_k^2}{tS^2}-\frac{56}{3(S+2\lambda_k^2)}
 =\frac{2}{tS}-\frac{2(S+2\lambda_k^2)}{tS^2}-\frac{56}{3(S+2\lambda_k^2)}
 \le\frac1S\left(\frac2t-8\sqrt{\frac7{3t}}\right),
\]
we obtain
\begin{equation}\label{eq:3.9}
	\begin{aligned}
		\Phi&\leq -4A+5B+\dfrac{\eta(t)}S\sum_{j}\big( \sum_i\lambda_i^2h_{iij}\bigl)^2,\\
	\end{aligned}
\end{equation}
where $\eta(t)$ is defined by
$$
\begin{aligned}
	\eta(t)&= 13(1+c_1)+\dfrac{2}{t}-8\sqrt{\dfrac{7}{3t}}\\
	&=24+\dfrac{3t-\sqrt{(22+3t)^2-400}}{2}+\dfrac{2}{t}
	-\dfrac{8\sqrt {21}}{3}\dfrac1{\sqrt{t}}.
\end{aligned}
$$

 Since $\dfrac{d\eta(t)}{dt}>0$ for  $2/7\le t\le3/10$,  $\eta(t)$
is an increasing function when $2/7\le t\le3/10$. Thus,
we have
\[
 0<\eta\left(\frac{2}{7}\right)\le\eta(t)\le\eta\left(\frac{3}{10}\right)<\frac{13}{4}.
\]
Combining
\eqref{eq:derivative-bound}
and \eqref{eq:3.9} , we obtain
\begin{equation}\label{eq:phi-bound}
 \Phi\le-4A+5B+\frac{\eta(t)}{3}(1+2\alpha)tSf
 \le-4A+5B+\frac{13}{12}(1+2\alpha)tSf.
\end{equation}
Substituting \eqref{eq:Phi-expanded} into \eqref{eq:phi-bound}
and collecting the terms containing $tSf$, we obtain
\begin{equation}\label{eq:second}
\begin{aligned}
 &-\left(\frac{65}{12}+\frac{13}{6}\alpha\right)tSf
 -\frac{2f_3}{S}\sum_{i,j,k}\lambda_i h_{ijk}^2\\
 &-\frac1{S^2}\left(\sum_{i,j,k}\lambda_i h_{ijk}^2\right)^2
 -\frac4{S(S-1)}\sum_k\lambda_k^2
       \left(\sum_i\lambda_i^2h_{iik}\right)^2
 \le -4A+5B .
\end{aligned}
\end{equation}

To estimate the three terms subtracted from
$-\left(65/12+13\alpha/6\right)tSf$ in \eqref{eq:second},
we denote their sum by
\begin{equation}\label{eq:U}
 U:=\frac{2f_3}{S}\sum_{i,j,k}\lambda_i h_{ijk}^2
  +\frac1{S^2}\left(\sum_{i,j,k}\lambda_i h_{ijk}^2\right)^2
  +\frac4{S(S-1)}\sum_k\lambda_k^2
       \left(\sum_i\lambda_i^2h_{iik}\right)^2.
\end{equation}
Thus \eqref{eq:second} can be written as
\begin{equation} \label{eq:second2}
 -\left(\frac{65}{12}+\frac{13}{6}\alpha\right)tSf-U\le -4A+5B.
\end{equation}

Firstly, we set
\begin{equation}\label{eq:normalized}
 \mu_i=\frac{\lambda_i}{\sqrt S},\qquad p=\frac{A-B}{S^2(S-1)}, \qquad \delta=(\max_i\mu_i-\min_i\mu_i)^2,
\end{equation}
then from \eqref{eq:ABbound} and Lemma~\ref{lem-spectral}, we obtain
\begin{equation}\label{eq:normalized-bounds}
p\le\frac\delta3(1-\alpha),\qquad0\le\delta\le2,
\end{equation}
and
\begin{equation}\label{eq:spectral-final}
 \delta<\frac65+\frac{15}{4}G,
\end{equation}
where
$$ G= \sum_i\mu_i^4-\left(\sum_i\mu_i^3\right)^2  =\frac{f_4}{S^2}-\frac{f_3^2}{S^3}=\frac{f}{S^2} \ge 0.$$

We now prove that $U$ is non-negative. For a
contradiction, we suppose that
\begin{equation*}
	U<0.
\end{equation*}
 Define
\[
 \theta=\frac1{1+2\alpha/5}, \quad \rho=\frac{23}{25}.
\]
Since $\alpha\ge0$, we have $0<\theta\le1$, and the definition of $\theta$ gives
\[
\theta\left(\frac{65}{12}+\frac{13}{6}\alpha\right)=\frac{65}{12}.
\]
Since $U<0$, the quantity $-\rho\theta U$ is strictly positive.
Multiplying \eqref{eq:second2} by the positive number $\rho\theta$ and discarding the positive term $-\rho\theta U$ give
\begin{equation}\label{eq:1}
	-\frac{65}{12}\rho tSf
	\le\rho\theta(-4A+5B).
\end{equation}
Adding \eqref{eq:1} to \eqref{eq:12} yields
\[
\begin{aligned}
 \left(\frac32-\frac{65}{12}\rho t\right)Sf
 &\le S(S-1)(S-2)+3(A-2B)+\rho\theta(-4A+5B)\\
 &=S(S-1)(S-2)+(3-4\rho\theta)A+(-6+5\rho\theta)B.
\end{aligned}
\]

Since $0<\rho\theta<1$, we have
\[
\begin{aligned}
 &(3-4\rho\theta)A+(-6+5\rho\theta)B\\
 &\quad=\left(4-\frac{13\rho\theta}{3}\right)(A-B)
       +\left(\frac{\rho\theta}{3}-1\right)(A+2B)\\
 &\quad\le\left(4-\frac{13\rho\theta}{3}\right)(A-B),
\end{aligned}
\]
where the last inequality follows from
\[
A+2B=\frac13\sum_{i,j,k}(\lambda_i+\lambda_j+\lambda_k)^2h_{ijk}^2\ge0.
\]
We therefore obtain
\begin{equation}\label{key1}
 \left(\frac32-\frac{65}{12}\rho t\right)Sf
\le\left(4-\frac{13\rho\theta}{3}\right)(A-B)+S(S-1)(S-2).
\end{equation}
Dividing \eqref{key1} by $tS^3$ and using
$p=(A-B)/(tS^3)$, $G=f/S^2$, and $S-1=tS$, we obtain
\begin{equation}\label{key1s}
\left(\frac{3}{2t}-\frac{65}{12}\rho\right)G
\le \left(4-\frac{13\rho\theta}{3}\right)p+2t-1.
\end{equation}

Since $2/7\le t\le3/10$ and $G\ge0$, we have
\begin{equation}\label{key2}
\left(\frac{3}{2t}-\frac{65}{12}\rho\right)G
\ge\left(5-\frac{23}{25}\cdot\frac{65}{12}\right)G=\frac1{60} G\ge0.
\end{equation}
On the other hand, since $0<\theta\le1$,
we have
$$4-\frac{13\rho\theta}{3}>0.$$
Then from \eqref{eq:normalized-bounds},  we obtain
\begin{equation}\label{key3}
	\begin{aligned}
\left(4-\frac{13\rho\theta}{3}\right)p+2t-1&\le \left(4-\frac{13\rho\theta}{3}\right)\frac{1-\alpha}{3}\delta+2t-1\\
&\le\frac3{25}\delta+2t-1\le-\frac4{25},
\end{aligned}
\end{equation}
where the second inequality follows from
\[
 \frac3{25}-\left(4-\frac{13\rho\theta}{3}\right)\frac{1-\alpha}{3}
 =\frac{600\alpha^2-541\alpha+130}{225(2\alpha+5)}>0.
\]
Combining \eqref{key1s}, \eqref{key2} and \eqref{key3}, we obtain
\begin{equation}\label{const1}
	0\le \left(\frac{3}{2t}-\frac{65}{12}\rho\right)G
	\le \left(4-\frac{13\rho\theta}{3}\right)p+2t-1\le-\frac4{25},
\end{equation}
which is impossible. Thus the assumption $U<0$ is false, and we have proved that
\[
U\ge0.
\]

Combining \eqref{eq:fourth} with $\theta$ times \eqref{eq:second} yields
\[
\begin{aligned}
\left(2-\frac{65}{12}t\right)Sf-2A+(1-\theta)U
\le{}&S(S-1)(S-2)+3(A-2B)\\
&+\theta(-4A+5B).
\end{aligned}
\]
Since $(1-\theta)U \ge 0$, $\theta\le1$ and $A+2B\ge0$, we obtain
\begin{equation}\label{eq:after-U}
\begin{aligned}
\left(2-\frac{65}{12}t\right)Sf
&\le S(S-1)(S-2)+(5-4\theta)A+(-6+5\theta)B\\
&=S(S-1)(S-2)+\frac{16-13\theta}{3}(A-B)+\frac{\theta-1}{3}(A+2B)\\
&\le S(S-1)(S-2)+\frac{16-13\theta}{3}(A-B).
\end{aligned}
\end{equation}

Dividing \eqref{eq:after-U} by $tS^3$ yields
\begin{equation}\label{eq:G-upper-final}
	\begin{aligned}
	\left(\frac2t-\frac{65}{12}\right)G
	&\le \frac{16-13\theta}{3}p+2t-1\\
	&\le \frac{53}{150}\delta+2t-1,
	\end{aligned}
\end{equation}
where the last inequality follows from
\[
\frac{53}{50}-\frac{16-13\theta}{3}(1-\alpha)
 =\frac{1600\alpha^2-532\alpha+45}{150(2\alpha+5)}>0.
\]

For fixed $\delta$, set
\[
r(t)=t\left(\frac{53}{150}\delta+2t-1\right).
\]
Differentiating $r(t)$ with respect to $t$, we obtain, for
$2/7\le t\le3/10$,
\[
r'(t)=\frac{53}{150}\delta+4t-1\ge\frac17>0,
\]
so $r$ is increasing. Multiplying \eqref{eq:G-upper-final} by $t$
gives $(2-65t/12)G\le r(t)$. Since $G\ge0$ and
$2-65t/12\ge3/8$, the monotonicity of $r$ and $\delta\le2$ yield
\[
\frac38G\le r(t)\le r\left(\frac3{10}\right)
 =\frac{53}{500}\delta-\frac3{25}
 \le\frac{23}{250}.
\]
Thus, we obtain
\begin{equation}\label{eq:G-bound}
0\le G\le\frac{92}{375}.
\end{equation}

Finally, adding \eqref{eq:fourth} and \eqref{eq:second} gives
\begin{equation}\label{eq:final-Sf}
\left[2-\left(\frac{65}{12}+\frac{13}{6}\alpha\right)t\right]Sf
\le(A-B)+S(S-1)(S-2).
\end{equation}
Dividing \eqref{eq:final-Sf} by $tS^3=S^2(S-1)$ and substituting
$p=(A-B)/(tS^3)$ and $G=f/S^2$ give
\begin{equation}\label{eq:lower-final}
\begin{aligned}
p+\frac{13}{6}\alpha G
&\ge1-2t+\left(\frac2t-\frac{65}{12}\right)G\\
&\ge\frac25+\frac54G,
\end{aligned}
\end{equation}
where the second follows from $t\le3/10$ and $G\ge0$.
Set
\[
d:=\frac25+\frac54G.
\]
From \eqref{eq:spectral-final} and \eqref{eq:G-bound}, we obtain
\[
\frac{\delta}{3}<d,\qquad
d-\frac{13}{6}G
 =\frac25-\frac{11}{12}G
 \ge\frac{197}{1125}>0.
\]
Since $p\le\delta(1-\alpha)/3$
and $0\le\alpha\le1$, from \eqref{eq:lower-final}, we obtain
\[
d\le p+\frac{13}{6}\alpha G
\le (1-\alpha)\frac\delta3+\frac{13}{6}\alpha G
< (1-\alpha)d+\alpha d=d,
\]
which is impossible. Therefore the interval under consideration is impossible, and the proof of the theorem is complete.
\end{proof}

\begin{proof}[Proof of Corollary~\ref{cor1}]
By Theorem~\ref{thm1}, it suffices to show that $S\le10/7$.
Suppose to the contrary that $S>10/7$. Since $R\ge-10/7$, by the Gauss equation \eqref{eq-2}
and the Cauchy--Schwarz inequality, we obtain
\[
 0<S-\frac{10}{7}\le H^2\le nS.
\]
Hence $u=H^{-2}$ is a smooth function bounded above on $M$.
From \eqref{eq-3}, we get
\[
\begin{aligned}
 L u
 &=-\frac{2}{H^3} L H+\frac{6}{H^4}|\nabla H|^2\\
 &=\frac{2(S-1)}{H^2}+\frac{6|\nabla H|^2}{H^4}
 \ge\frac{2(S-1)}{nS}>0.
\end{aligned}
\]
Since $S$ is constant, \eqref{eq-1} also implies that the Ricci
curvature is bounded from below. Applying Lemma~\ref{lem:maximum} to $u$,
we obtain a sequence $\{p_j\}\subset M$ such that
\[
 0<\frac{2(S-1)}{nS}
 \le\limsup_{j\to\infty} L u(p_j)\le0,
\]
which leads to a contradiction. Thus $S\le10/7$, and the proof is complete.
\end{proof}

{\bf Acknowledgements:} Authors are supported by the grant No. 12071028 of NSFC.

\end{document}